\documentclass[letterpaper, 10 pt, conference]{ieeeconf}
\IEEEoverridecommandlockouts                              % This command is only
\usepackage{graphicx}
\usepackage{bm}
\usepackage{amsmath,amssymb,latexsym}
\usepackage{cite}

\usepackage{etex}
\usepackage{tikz}
\usetikzlibrary{shapes,arrows}
\usetikzlibrary{arrows}
\usetikzlibrary{decorations.pathreplacing}

\newtheorem{theorem}{Theorem}

\newtheorem{lemma}{Lemma}

\newtheorem{corollary}{Corollary}
\newtheorem{remark}{Remark}

\ifnum\pdfshellescape=1
\fi

\newcommand{\E}[1]{\mathbb{E}\left\{ #1\right\}}
\renewcommand{\exp}[1] {\mathrm{exp}\left( #1\right)}
\usepackage{caption}
\usepackage{subcaption}
\title{Mean Square Stabilization of Vector LTI Systems over Power Constrained Lossy Channels}

 \author{Liang Xu, Yilin Mo and Lihua Xie % <-this % stops a space
 \thanks{Liang Xu, Yilin Mo and Lihua Xie are with School of Electrical and Electronic Engineering, Nanyang Technological University, Singapore.
         Email: {\tt\small lxu006@e.ntu.edu.sg,  ylmo@ntu.edu.sg, elhxie@ntu.edu.sg}}}

\begin{document} \maketitle

\begin{abstract}
    This paper studies the mean square stabilization problem of vector
    LTI systems over power constrained lossy channels. The
    communication channel is with packet dropouts, additive noises and
    input power constraints. To overcome the difficulty of optimally
    allocating channel resources among different sub-dynamics,
    schedulers are designed with time division multiplexing of
    channels. An adaptive TDMA (Time Division Multiple Access)
    scheduler is proposed first, which is shown to be able to achieve
    a larger stabilizability region than the conventional TDMA
    scheduler, and is optimal under some special cases. In particular,
    for two-dimensional systems, an optimal scheduler is designed,
    which provides the necessary and sufficient condition for mean
    square stabilization.
\end{abstract}

\section{Introduction}

For ease of installation and maintenance, wireless communication has
potentially wide applications in control systems. Due to change of
environments, fading and additive noises are unavoidable in wireless
communications, which motivates the study on how they affect the
stability and performance of control systems.

Traditionally, fading and additive communication noises are studied
separately. For example, \cite{BraslavskyJH2007TAC,
  FreudenbergJS2010TAC} study the stabilization problem of linear
systems controlled over power constrained AWGN channels. They show
that there exists a kind of channel capacities which is related to the
unstable eigenvalues of the linear system, above which there exist no
stabilizing feedback control strategies. This is parallel to the
data-rate theorem in~\cite{NairGN2004SIAM}, which establishes a
critical data rate for a rate limited communication channel below
which the system cannot be stabilized. Similarly, for pure fading
channels,\cite{EliaN2005SCL} shows that there exists a mean square
capacity that determines the stabilizability of the open-loop system.
However, since fading and additive noises exist simultaneously in
wireless communication systems, it is practical to consider them as a
whole. Previously, we have derived necessary and sufficient
stabilizability conditions for LTI systems controlled over power
constrained fading channels~\cite{XuLiang2015CDC}. The strategies
derived there are shown to be optimal for scalar systems. While for
vector systems, generally there exists a gap between the necessary
condition and the sufficient condition.

For vector systems, the difficulty is how to optimally allocate
channel resources among different sub-systems. Similar difficulties
are also encountered in networked control over rate limited
communication channels. It is shown in~\cite{MineroP2009TAC} that the
main difficulty in stabilizing a multi-dimensional system over random
digital channels consists of allocating optimally the bits to each
unstable sub-system. They introduce a rate allocation vector which
determines the fraction of rates that is allocated to each sub-system
to solve this problem. Generally, the number of bits to each state
variable is proportional to the magnitude of the corresponding
unstable mode~\cite{YouKeyou2011TACmarkovian}. The stabilizability
region achieved by this method is a convex hull, which can be
conservative even for two-dimensional systems. This rate vector
allocation scheme for digital channels essentially implies a FDMA
(Frequency Division Multiple Access) strategy for applications to
analogy channels. However, FDMA schemes are difficult to design and
analyze. In this paper, we propose an adaptive TDMA communication
protocol, which achieves a similar effect as the rate allocation
vector used in\cite{MineroP2009TAC}\cite{YouKeyou2011TACmarkovian}.
Moreover, we show that the optimal allocation is time-varying, which
contrasts with the constant rate vector allocation. Based on this
analysis, an optimal scheduler is proposed for two-dimensional
systems, which can provide the necessary and sufficient
stabilizability condition.

This paper is organized as follows: in Section II, the problem is
formulated and preliminaries are provided. Section III illustrates the
adaptive TDMA scheduler design and its stability analysis. An optimal
scheduler is proposed and analyzed for two-dimensional systems in
Section IV.\ This paper ends with some concluding remarks in Section
V.

\section{Problem Formulations and Preliminaries}

This paper studies the following single-input discrete-time linear
system
\begin{equation}
    \label{eq.LtiDynamics}
    x_{t+1}=A x_{t}+B u_{t}
\end{equation}
where $x\in \mathbb{R}^N$ is the system state, $u \in \mathbb{R}$ is
the control input and $(A,B)$ is stabilizable. Without loss of
generality, we can assume that $A$ is in the real Jordan canonical
form and all its eigenvalues are either on or outside of the unit
disk. Let $\lambda_1, \ldots, \lambda_d$ be the distinct unstable
eigenvalues (if $\lambda_i$ is complex, we exclude its complex
conjugates $\lambda_i^*$ from this list) of $A$ with
$|\lambda_1|\ge |\lambda_2|\ge \ldots \ge |\lambda_d|$. Let $m_i$ be
the algebraic multiplicity of each $\lambda_i$. Then $A$ has the block
diagonal structure
$A=\mathrm{diag}(J_1,\ldots,J_d)\in \mathbb{R}^{N\times N}$, where the
block $J_i\in \mathbb{R}^{\mu_i\times \mu_i}$ with
\begin{equation*}
    \mu_i =\left\{ \begin{matrix}
            m_i & \mathrm{if} \;\; \lambda_i \in \mathbb{R}\\
            2m_i & \mathrm{otherwise}
        \end{matrix} \right.
\end{equation*}
The initial value $x_0=[x_{1,0}, \ldots, x_{N,0}]$ is randomly
generated from a Gaussian distribution with zero mean and bounded
covariance matrix ${ \Sigma_{x_0}}>0$. The system state $x_t$ is
observed by a sensor and then encoded and transmitted to the
controller through a power constrained lossy channel with
\begin{equation}
    \label{eq:channel}
    r_t=\gamma_t s_t+n_t
\end{equation}
where $s_t$ denotes the channel input; $r_t$ represents the channel
output; $\{\gamma_t\}$ models the i.i.d.\ packet drop process with
Bernoulli distribution $\text{Pr}(\gamma_t=0)=\epsilon$,
$\text{Pr}(\gamma_t=1)=1-\epsilon$ and $\{n_t\}$ is the additive white
Gaussian communication noise with zero-mean and bounded variance
$\sigma_n^2$. The channel input $s_t$ must satisfy an average power
constraint, i.e., $\E{s_t^2}\le P$. We also assume that
$x_0, \gamma_0, n_0, \gamma_1, n_1, \ldots$ are independent. In the
paper, it is assumed that after each transmission, the instantaneous
value of $\gamma_t$ is known to the decoder, which is reasonable for
slow-varying channels with channel
estimation~\cite{GoldsmithA2005Book}. Besides, there exists a feedback
link that communicates $r_{t-1}$ and $\gamma_{t-1}$ from the channel
output to the channel input. The feedback configuration among the
plant, the sensor and the controller, and the channel encoder/decoder
structure are depicted in Fig~\ref{fig.NCS}.

\begin{figure}[htpb]
    \centering
    \includegraphics[width=0.4\textwidth]{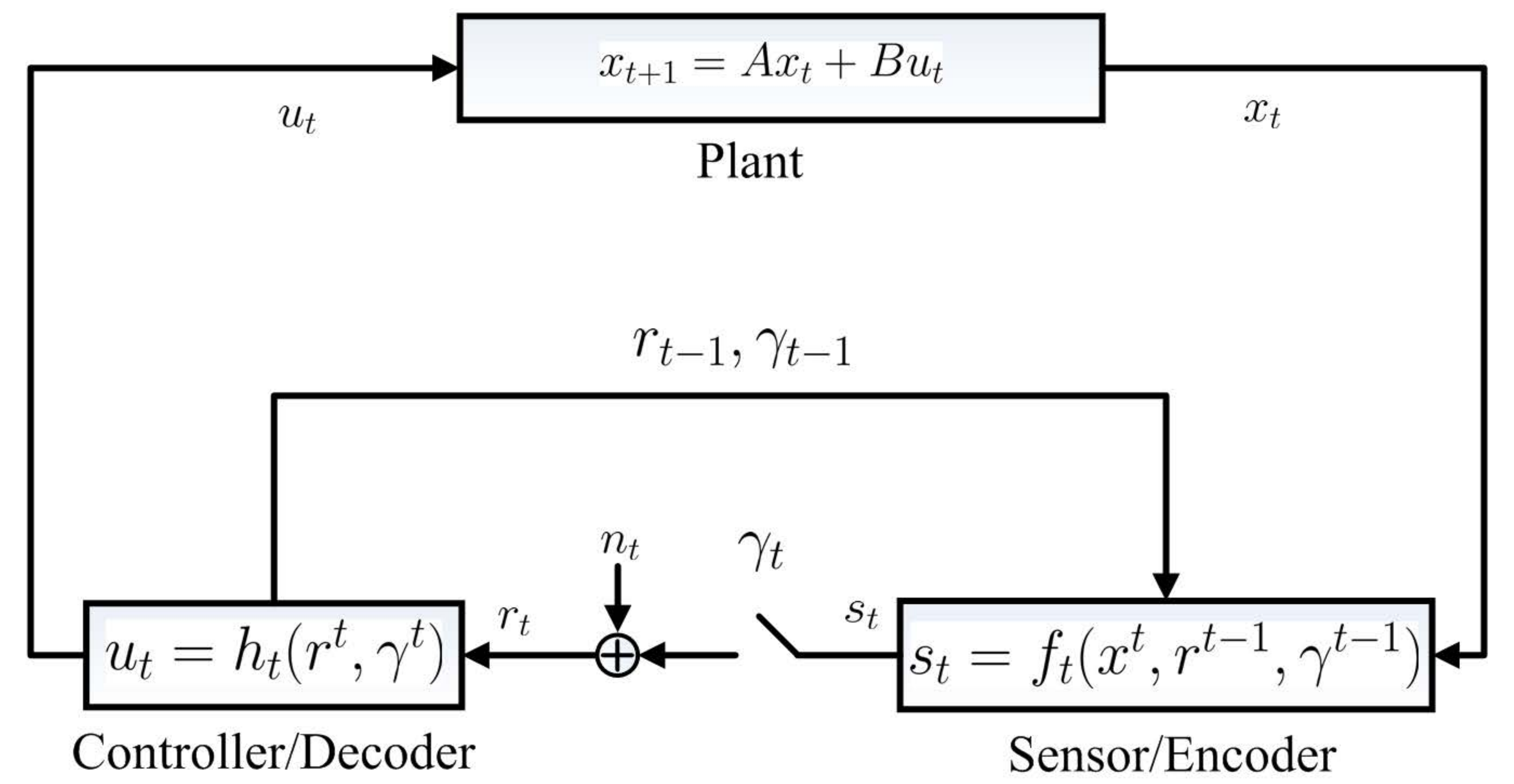}\\
    \caption{Network control structure over a power constrained lossy
      channel}
    \label{fig.NCS}
\end{figure}
In this paper, we try to find conditions on the
channel~\eqref{eq:channel} such that there exists a pair of
encoder/decoder $\{f_t\}, \{h_t\}$ that can mean square stabilize the
LTI dynamics~\eqref{eq.LtiDynamics}, i.e., to render
$\mathrm{lim}_{t \rightarrow \infty} \E{x_tx_t'}=0$. If we define
$\delta= \frac{\sigma_n^2}{\sigma_n^2+P}$, the necessary condition and
the sufficient condition to ensure mean square stabilizability
in~\cite{XuLiang2015CDC} are first recalled in the lemma below.
\begin{lemma}
    There exists an encoder/decoder pair $\{f_t\}, \{h_t\}$, such that
    the LTI dynamics~\eqref{eq.LtiDynamics} can be stabilized over the
    communication channel~\eqref{eq:channel} in mean square sense if
    \begin{equation}
        \label{eq:TdmaSufficiency}
        \sum _{i=1}^d \mu_i\mathrm{ln} |\lambda_i| < - \frac{1}{2} \mathrm{ln}(\epsilon+(1-\epsilon) \delta)
    \end{equation}
    and only if
    $(\mathrm{ln}|\lambda_1|, \ldots, \mathrm{ln}|\lambda_d|) \in
    \mathbb{R}^{d}$
    satisfy that for all $v_i \in \{0, \ldots, m_i\}$ and
    $i\in \mathcal{U} = \{1, \ldots, d\}$
    \begin{equation}
        \label{eq.Necessity}
        \sum_{i \in \mathcal{U}} a_i v_i \mathrm{ln}|\lambda_i|< - \frac{v}{2} \mathrm{ln}  (\epsilon+(1-\epsilon) \delta^{\frac{1}{v}})
    \end{equation}
    where $v=\sum_{i\in \mathcal{U}} a_iv_i$, and $a_i=1$ if
    $\lambda_i \in \mathbb{R}$, and $a_i=2$ otherwise.
\end{lemma}

The sufficient condition~\eqref{eq:TdmaSufficiency} is achieved by
using a TDMA strategy, where each sub-dynamics is allocated a fixed
period to use the channel. In the following section, we propose an
adaptive TDMA communication scheme for $N$-dimensional systems which
achieves a less conservative result than~\eqref{eq:TdmaSufficiency}.

\section{Adaptive TDMA Scheme for $N$-dimensional Systems}

Before stating the communication scheme, the following lemma is listed
first, which is instrumental to the protocol design.
\begin{lemma}[\cite{KumarU2014TAC}]
    \label{lemma:estimationThenControl}
    If there exists an estimation scheme $\hat{x}_t$ for the initial
    system state $x_0$, such that the estimation error
    $e_t=\hat{x}_t-x_0= [e_{1,t}, e_{2,t},\ldots, e_{N,t}]$ satisfies
    the following property,
    \begin{eqnarray}
      \label{eq:estimationThenControlRequirement}
      \E{ e_t }=0 \label{eq:estimationThenControlRequirement1} \\
      \lim_{t\rightarrow \infty} A^t \E{ e_t e_t' } (A')^t=0  \label{eq:estimationThenControlRequirement2}
    \end{eqnarray}
    the system~\eqref{eq.LtiDynamics} can be mean square stabilized by
    the controller
    \begin{equation}
        \label{eq:controller}
        u_t=K\left( A^t \hat{x}_t+ \sum_{i=1}^t A^{t-i} Bu_{i-1} \right)
    \end{equation}
    with $K$ being selected such that $A+BK$ is stable.
\end{lemma}

\subsection{Encoder and Decoder Design}

In light of Lemma~\ref{lemma:estimationThenControl}, we only need to
design a communication protocol to
guarantee~\eqref{eq:estimationThenControlRequirement1}
and~\eqref{eq:estimationThenControlRequirement2}. The transmission
protocol used in this paper contains three parts: the encoder, the
decoder and the scheduler. The structure of the transmission protocol
is illustrated in Fig.~\ref{fig:scheduler}.
\begin{figure}[htpb]
    \centering
    \includegraphics[width=0.5\textwidth]{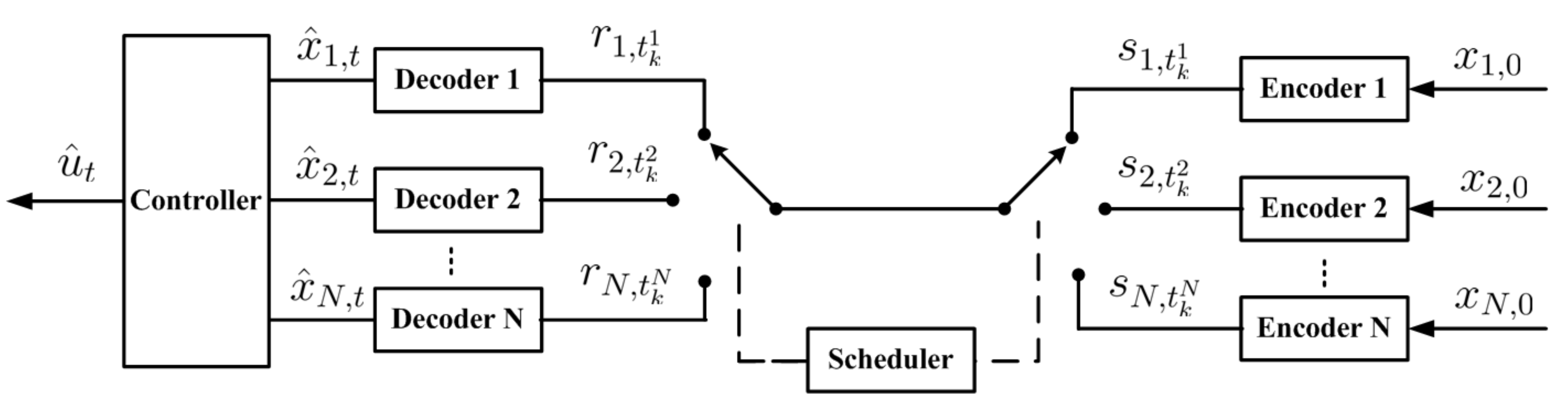}\\
    \caption{Transmission protocol configuration}
    \label{fig:scheduler}
\end{figure}

The $i$-th encoder/decoder pair is designed to transmit the
information corresponding to $x_{i,0}$. The controller maintains an
array
$\hat{x}_t=[\hat{x}_{1,t}, \hat{x}_{2,t}, \ldots, \hat{x}_{N,t}]$ that
represents the most recent estimation of $x_0$, which is set to $0$
for $t=0$. When the information about $x_{i,0}$ is transmitted, only
$\hat{x}_{i, t}$ is updated at the controller side. There is one
scheduler that determines which encoder/decoder pair should use the
channel. Denote $t^i_k$ to be the time when the $i$-th encoder/decoder
pair is scheduled to use the channel for its $k$-th transmission.
$t^i_k$ is thus updated only at the scheduled time.

The encoder $i$ is designed as
\begin{equation}
    \label{eq:encoder}
    \begin{aligned}
        s_{i, t^i_0}  &=\sqrt{\frac{P}{\sigma_{x_{i,0}}^2}} x_{i,0}\\
        s_{i,t^i_k}&=\sqrt{\frac{P}{\sigma^2_{e_{i, t^i_{k-1}}}}}\left( \hat{x}_{i, t^{i}_{k-1}} - x_{i,0} \right), \;\; k\ge 1\\
    \end{aligned}
\end{equation}
where $ \hat{x}_{i,t^i_{k-1}}$ denotes the estimate of $x_{i,0}$ at
the time $t^i_{k-1}$.

The decoder $i$ satisfies
\begin{equation}
    \label{eq:decoder}
    \begin{aligned}
        \hat{x}_{i, t_0^i} & =\sqrt{\frac{\sigma_{x_{i,0}}^2}{P}}r_{i, t_0^i}\\
        \hat{x}_{i, t_k^i}&=\hat{x}_{i, t^i_{k-1}}-\frac{\E{r_{i,
              t^i_k} e_{i, t^i_{k-1}}|\gamma_{ t_k^i}}}{\E{r_{i,
              t_k^i}^2|\gamma_{t_k^i}}}r_{i, t_k^i}, \;\; k\ge 1
    \end{aligned}
\end{equation}
with $\sigma^2_{e_{i, t}}$ representing the variance of $e_{i, t}$.

Similar to the analysis in~\cite{XuLiang2015CDC}, we can show that
under the encoder~\eqref{eq:encoder}, and the
decoder~\eqref{eq:decoder},~\eqref{eq:estimationThenControlRequirement1}
always holds and
$\E{e_{i,t}^2}=\E{\delta^{n_i^t}} \E{ e_{i, t_0^i}^2}$ with $n_i^t$
denoting the total number of successful packet receptions by the
$i$-th decoder by time $t$, which is determined both by the scheduler
and the stochastic packet drop process. Thus to
guarantee~\eqref{eq:estimationThenControlRequirement2}, generally we
should design schedulers to ensure
$\lim_{t\rightarrow \infty} \E{ \lambda_i^{2t}\delta^{n_i^t}}=0$ for
all $i=1, \ldots, N$. In the following, an adaptive TDMA scheduler is
designed and its stability property is proved.

\subsection{Scheduler Design}

Different from the fixed period transmission in the TDMA scheduler
used in~\cite{XuLiang2015CDC}, the adaptive TDMA scheduler used here
is adapted to the packet drop process. It switches the transmission
only if the packet is received for certain times. By using information
of the packet drop process, we may expect to achieve a larger
stabilizability region. The scheduler is described as below.
\begin{center}
    \begin{tabular}{ p{0.45\textwidth} }
      \hline
      \textbf{Algorithm 1}: Adaptive TDMA Scheduler for $N$-dimensional Systems\\
      \hline
      \begin{itemize}
        \item The first encoder/decoder pair is scheduled to used the
          channel, until the transmissions succeed for $n_1$ times.
        \item The second encoder/decoder pair is scheduled to use the
          channel, until the transmissions succeed for $n_2$ times.
        \item \ldots
        \item The $N$-th encoder/decoder pair is scheduled to use the
          channel, until the transmissions succeed for $n_N$ times.
        \item Repeat.
      \end{itemize}\\
      \hline
    \end{tabular}
\end{center}
The transmission scheduling is depicted in
Fig.~\ref{fig.adaptiveTDMAScheduler}, in which $T_{k}^i$ denotes the
time period for the $i$-th encoder/decoder pair to achieve $n_i$
successful transmissions during the $k$-th round; $T_k^t$ denotes the
total time period to complete the $k$-th round transmission, i.e.
$T_k^t=\sum_{i=1}^NT_k^i$. It is clear that $T_k^i$ is independent
with $T_k^j$, and $T_i^t$ is independent with $T_j^t$ for any
$i, j, k$.
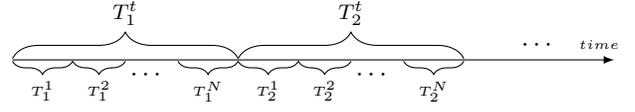
\begin{figure}
    \begin{center}
        \begin{tikzpicture}[scale=1, >=latex]
            \draw[->] (0,0) -- (8, 0); \node at (7.8, 0.2) {\tiny
              ${time}$}; \draw
            [decorate,decoration={brace,amplitude=10pt},xshift=0pt,yshift=0.5pt]
            (0,0) -- (3,0) node [black,midway,yshift=0.6cm]
            {\footnotesize $T_1^t$}; \draw
            [decorate,decoration={brace,amplitude=10pt},xshift=0pt,yshift=0.5pt]
            (3,0) -- (6,0) node [black,midway,yshift=0.6cm]
            {\footnotesize $T_2^t$}; \draw
            [decorate,decoration={brace,amplitude=5pt,mirror,raise=0pt},yshift=-0.5pt]
            (0,0) -- (0.8,0) node [black,midway,yshift=-0.4cm] {\tiny
              $ T^1_1$}; \draw
            [decorate,decoration={brace,amplitude=5pt,mirror,raise=0pt},yshift=-0.5pt]
            (0.8,0) -- (1.5,0) node [black,midway,yshift=-0.4cm]
            {\tiny $T^2_1$}; \node at (1.8, - 0.2) {$\cdots$}; \draw
            [decorate,decoration={brace,amplitude=5pt,mirror,raise=0pt},yshift=-0.5pt]
            (2.2,0) -- (3,0) node [black,midway,yshift=-0.4cm] {\tiny
              $T^N_1$}; \draw
            [decorate,decoration={brace,amplitude=5pt,mirror,raise=0pt},yshift=-0.5pt]
            (3,0) -- (3.8,0) node [black,midway,yshift=-0.4cm] {\tiny
              $ T^1_2$}; \draw
            [decorate,decoration={brace,amplitude=5pt,mirror,raise=0pt},yshift=-0.5pt]
            (3.8,0) -- (4.5,0) node [black,midway,yshift=-0.4cm]
            {\tiny $T^2_2$}; \node at (4.8, - 0.2) {$\cdots$}; \draw
            [decorate,decoration={brace,amplitude=5pt,mirror,raise=0pt},yshift=-0.5pt]
            (5.2,0) -- (6,0) node [black,midway,yshift=-0.4cm] {\tiny
              $T^N_2$}; \node at (7, 0.2) {$\cdots$};
        \end{tikzpicture}
    \end{center}
    \caption{Transmissions with the Adaptive TDMA scheduler}
    \label{fig.adaptiveTDMAScheduler}
\end{figure}

\begin{remark}
    Here we assume the encoder and the decoder are both aware of the
    scheduling algorithm. Since the switching among transmissions in
    our designed schedulers relies on the packet drop process, and
    there exists a feedback channel that acknowledges the packet drop,
    the encoder and the decoder are both aware of when to switch
    transmissions and what is the encoder/decoder pair that
    corresponds to the current channel use. This implies an implicit
    consensus among the encoder and the decoder. Thus we do not need
    to consider the coordination problem between the encoders and the
    decoders.
\end{remark}

\subsection{Stability Results}
Before stating the result, the following lemma is needed.
\begin{lemma}[The Binomial Theorem]
    \label{lemma.BinomialTheorem}
    \begin{gather*}
        \sum_{k=0}^t \binom{t}{k}x^k y^{t-k}=(x+y)^t\\
        \sum_{k=0}^{\infty} \binom{t+k-1}{t-1}x^k=\frac{1}{(1-x)^t},
        \quad (|x|<1)
    \end{gather*}
\end{lemma}
\begin{theorem}
    If there exist $\alpha_i>0$ with $\sum_{i=1}^d\alpha_i=1$, such
    that
    \begin{equation}
        \label{eq:AdaptiveTDMASufficiency}
        \ln |\lambda_i| < -\frac{1}{2} \ln \left( \epsilon+(1-\epsilon)\delta^{\frac{\alpha_i}{\mu_i}} \right)
    \end{equation}
    for all $i=1,\ldots, d$, the LTI dynamics~\eqref{eq.LtiDynamics}
    can be stabilized over the communication
    channel~\eqref{eq:channel} in mean square sense with the
    encoder~\eqref{eq:encoder}, the decoder~\eqref{eq:decoder} and the
    scheduler described in Algorithm 1.
\end{theorem}

\begin{proof} Here we only consider the case that
    $\lambda_1,\ldots, \lambda_d$ are real and $m_i=1$. We can easily
    extend the analysis to other cases by following a similar line of
    arguments as in~\cite{KumarU2014TAC} and the Section 2.3.1.2
    in~\cite{ComoG2014Book}.

    Since the erasure process is i.i.d., $\{T_k^i\}$ is i.i.d.\ for
    all $i=1, 2,\ldots, N$ with the probability distribution
    \begin{equation}
        \label{eq.pdfT1}
        \text{Pr}(T^i_k=n_i+l)= \binom{n_i+l-1} {n_i-1}(1-\epsilon)^{n_i}\epsilon^l
    \end{equation}
    with $l=0, 1, 2,\ldots$. In light of
    Lemma~\ref{lemma.BinomialTheorem}, we have that
    \begin{align}
      \E{\lambda_i^{2T^j_k}}&=\sum_{l=0}^{\infty} \lambda_i^{2(n_j+l)}  \binom {n_j+l-1}{n_j-1}  (1-\epsilon)^{n_j}\epsilon^l \nonumber \\
                            &=\lambda_i^{2n_j} \frac{(1-\epsilon)^{n_j}}{(1-\epsilon\lambda_i^2)^{n_j}} \label{eq.expectationExpTj}
    \end{align}
    Since $T_k^j$ is independent with $T_k^i$ for all
    $i, j \in \{ 1, 2, \ldots, N\}$, we have
    \begin{align}
      &\E {\lambda_i^{2( \sum_{j=1}^NT_k^j )} \delta^{n_i}}=\prod_{j=1}^N \E{\lambda_i^{2T_k^j}} \delta^{n_i} \nonumber \\
      % &=\left( \lambda_i^{2} \frac{(1-\epsilon
      %     )}{(1-\epsilon\lambda_i^2)}\right)^{\sum_{j=1}^N n_j}
      % \delta^{n_i} \nonumber \\
      &=\left( \lambda_i^{2}
        \frac{(1-\epsilon)}{(1-\epsilon\lambda_i^2)}\delta^{\frac{n_i}{\sum_{j=1}^N n_j}}\right)^{\sum_{j=1}^N n_j} \label{eq.expontationExpTtDelta}
    \end{align}
    Besides, if we define $T_0^t=0$, we have
    \begin{align*}
      &\E{ \sum_{t=1}^{\infty} \lambda_i^{2t} \delta^{n_i^t} }  \le \sum_{k=0}^{\infty}\E{ \sum_{j=1}^{T_{k+1}^t-1} \lambda_i^{2(T_0^t +\ldots +T_k^t +j )}  \delta^{k n_i}}\\
      & =  \sum_{k=0}^{\infty}  \E{   \frac{\lambda_{i}^{2T_{k+1}^t}- \lambda_{i}^2}{\lambda_{i}^2-1}  } \E{ \lambda_i^{2T_1^t}  \delta^{ n_i}}^{k} \\
      & =  \sum_{k=0}^{\infty}  \E{   \frac{\lambda_{i}^{2T_{k+1}^t}- \lambda_{i}^2}{\lambda_{i}^2-1}  }  \left( \lambda_i^{2} \frac{(1-\epsilon)}{(1-\epsilon\lambda_i^2)}\delta^{\frac{n_i}{\sum_{j=1}^N n_j}}\right)^{k (\sum_{j=1}^N n_j)}
    \end{align*}
    In view of~\eqref{eq.expectationExpTj}, we know that
    $\E{ \frac{\lambda_{i}^{2T_{k+1}^t}-
        \lambda_{i}^2}{\lambda_{i}^2-1} }$
    is bounded. Moreover, if~\eqref{eq:AdaptiveTDMASufficiency} holds,
    we can always find $n_j$s such that
    \begin{equation*}
        \left( \lambda_i^{2} \frac{(1-\epsilon)}{(1-\epsilon\lambda_i^2)}\delta^{\frac{n_i}{\sum_{j=1}^N n_j}}\right)^{\sum_{j=1}^N n_j} <1
    \end{equation*}
    for all $i=1, 2,\ldots, N$, which further implies
    $\E{ \sum_{t=1}^{\infty}\lambda_i^{2t} \delta^{n_i^t}}<\infty$.
    Thus $\lim_{t\rightarrow } \E{ \lambda_i^{2t} \delta^{n_i^t}}=0$
    for all $i=1, \ldots, N$. In light of
    Lemma~\ref{lemma:estimationThenControl}, the result can be proved.
\end{proof}

\begin{remark}
    \label{remark:adaptiveLargerThanConventional}
    The sufficiency~\eqref{eq:TdmaSufficiency} achieved with the TDMA
    scheduler can be alternative formulated as if there exist
    $\alpha_i>0$ and $\sum_{i=1}^d\alpha_i=1$, such that
    \begin{equation*}
        \ln |\lambda_i| <-\frac{\alpha_i}{2\mu_i} \ln (\epsilon + (1-\epsilon) \delta)
    \end{equation*}
    for all $i=1,2,\ldots, d$, the system~\eqref{eq.LtiDynamics} can
    be mean square stabilized. Since
    \begin{equation*}
        -\frac{\alpha_i}{2\mu_i} \ln (\epsilon + (1-\epsilon) \delta)< -
        \frac{1}{2} \mathrm{ln} \left( \epsilon+(1-\epsilon)
            \delta^{\frac{\alpha_i}{\mu_i} } \right)
    \end{equation*}
    any $\lambda_i$ that satisfies~\eqref{eq:TdmaSufficiency} must
    also satisfy~\eqref{eq:AdaptiveTDMASufficiency} with the same
    $\alpha_i$, which implies that the adaptive TDMA scheduler
    achieves a larger stabilizability region than the TDMA scheduler.
\end{remark}

When all the strictly unstable eigenvalues have the same magnitude, we
can show that the sufficient condition
\eqref{eq:AdaptiveTDMASufficiency} coincides with the necessary
condition~\eqref{eq.Necessity}. The result is given in the following
corollary.
\begin{corollary}
    \label{corollary.equalMagnitude}
    If $\exists d_u\le d$, such that
    $|\lambda_1|=\ldots =|\lambda_{d_u}|=\lambda>1$ and
    $|\lambda_{d_u+1}|=\ldots = |\lambda_{d}|=1$, there exists an
    encoder/decoder pair $\{f_t\}, \{h_t\}$, such that the LTI
    dynamics~\eqref{eq.LtiDynamics} can be stabilized over the
    communication channel~\eqref{eq:channel} in mean square sense if
    and only if
    \begin{equation*}
        \ln \lambda< -\frac{1}{2} \ln \left(\epsilon+ (1-\epsilon) \delta^{\frac{1}{\mu_1+\ldots +\mu_{d_u}}}\right)
    \end{equation*}
\end{corollary}

When the strictly unstable eigenvalues are with distinct magnitudes,
generally there exists a gap between the necessary stabilizability
condition~\eqref{eq.Necessity} and the sufficient stabilizability
condition~\eqref{eq:AdaptiveTDMASufficiency} that can be achieved by
the adaptive TDMA scheduler. In the following, we propose an optimal
scheduler design for two-dimensional systems, specifically with
distinct magnitudes, that can stabilize all the eigenvalue pairs in
the necessary stabilizability region.

\section{Optimal Scheduler for Two-Dimensional Systems}
Since when the eigenvalues are with equal magnitudes, the adaptive
TDMA scheduler is optimal. Without loss of generality, in this section
we assume that
$A= \left[ \begin{smallmatrix} \lambda_1 & 0\\ 0 &
        \lambda_2 \end{smallmatrix} \right]$
with $\lambda_1, \lambda_2\in \mathbb{R}$ and
$|\lambda_1| >|\lambda_2|>1$ and propose an optimal scheduler design
for such systems. In view of Lemma~\ref{lemma:estimationThenControl}
and the encoder and decoder~\eqref{eq:encoder}~\eqref{eq:decoder}, we
should design schedulers to ensure that under stochastic packet
dropouts
\begin{equation*}
    \lim_{t\rightarrow \infty} \E{\lambda_1^{2t}\delta^{n_1^t}}=0,\;\; \lim_{t\rightarrow \infty} \E{\lambda_2^{2t}\delta^{n_2^t}}=0
\end{equation*}
or equivalently
\begin{equation}
    \label{eq:NonOracleobjective}
    \lim_{t\rightarrow \infty} \E{\lambda_1^{2t}\delta^{n_1^t}+ \lambda_2^{2t}\delta^{n_2^t}}=0
\end{equation}
A critical condition to ensure~\eqref{eq:NonOracleobjective} is that
the minimal value of
$\E{\lambda_1^{2t}\delta^{n_1^t} +\lambda_2^{2t}\delta^{n_2^t}}$
converges to zero asymptotically. Thus the scheduler should be
designed to optimally allocate $n_1^t$ and $n_2^t$ to minimize
$ \lambda_1^{2t}\delta^{n_1^t}+\lambda_2^{2t}\delta^{n_2^t}$. The
optimal allocation of $n_1^t$ and $n_2^t$ should satisfy that
\begin{equation}
    \label{eq:goal}
    n_2^t=n_1^t+2t \frac{\ln |\lambda_1|-\ln |\lambda_2|}{\ln \delta}
\end{equation}
which is obtained by requiring
$ \lambda_1^{2t}\delta^{n_1^t}=\lambda_2^{2t}\delta^{n_2^t}$. In the
following, we propose a scheduler design which enforces $n_1^t$ and
$n_2^t$ to satisfy~\eqref{eq:goal} when $t$ is sufficiently large in
the presence of stochastic packet dropouts. Then we may expect that
the scheduler is optimal.

\subsection{Optimal Scheduler Design}

\begin{center}
    \begin{tabular}{ p{0.45\textwidth} }
      \hline
      \textbf{Algorithm 2}: Optimal Scheduler for Two-dimensional Systems\\
      \hline
      \begin{itemize}
        \item In the $k$-th round, the first encoder/decoder pair is
          scheduled to use the channel until the transmissions succeed
          for $n_1$ times. Denote the time period to achieve this
          object as $T_k^1$.
        \item \begin{itemize}
            \item If
              \begin{equation}
                  \label{eq.swithchingCondition}
                  n_1+2T_k^1 \frac{\ln |\lambda_1|-\ln |\lambda_2|}{\ln \delta}>0
              \end{equation}
              the second encoder/decoder pair is scheduled to use the
              channel until the transmissions succeed for $n_{2,k}$
              times with
              \begin{equation}
                  \label{eq:stoppingcondition}
                  n_{2, k}>n_1+2(T_k^1+T_k^2)\frac{\ln|\lambda_1|-\ln |\lambda_2|}{\ln \delta}
              \end{equation}
              where $T_k^2$ denotes the time period of achieving this
              object.
            \item Otherwise, set $T_k^2=0$ and do not conduct any
              transmissions.
          \end{itemize}
        \item Repeat.
      \end{itemize}\\
      \hline
    \end{tabular}
\end{center}

Thus $T^1_k$ has the probability distribution~\eqref{eq.pdfT1} with
$i=1$. Let $T_k^t$ denote the total time used to complete the $k$-th
round transmission, i.e., $T_k^t=T_k^1+T_k^2$. It is clear that
$T_i^t$ is independent with $T_j^t$ and $n_{2,i}$ is independent with
$n_{2,j}$ for any $i,j$. The switching
condition~\eqref{eq.swithchingCondition} implies that if
\begin{equation*}
    T_k^1\le T^c := \frac{n_1 \ln \delta}{2\left( \ln |\lambda_2|
          -\ln |\lambda_1| \right)}
\end{equation*}
after finishing transmitting the estimate corresponding to $x_{1,0}$,
the estimate corresponding to $x_{2,0}$ can be transmitted. Otherwise,
the algorithm continues to use the channel to transmit the estimate
corresponding to $x_{1,0}$. Besides, it is clear that $T_k^2$ is a
stopping time when $T_k^1\le T^c$. Moreover $T_k^2$ is bounded when
$T_k^1\le T^c$ due to the fact that $|\lambda_2| < |\lambda_1|$.
Hence, even if all transmissions fail, we still have
$T_k^1+T_k^2 \leq T^c$, which means $T_k^2$ is bounded.

\subsection{Stability Results}
The result is stated in the following theorem.
\begin{theorem}
    \label{theorem:TwoDimensionalSystems}
    Suppose
    $A= \left[ \begin{smallmatrix} \lambda_1 & 0\\ 0 &
            \lambda_2 \end{smallmatrix} \right]$
    with $\lambda_1, \lambda_2\in \mathbb{R}$ and
    $|\lambda_1| >|\lambda_2|>1$, the LTI
    dynamics~\eqref{eq.LtiDynamics} is mean square stabilizable over
    the power constrained lossy channel~\eqref{eq:channel} if and only
    if
    \begin{gather}
        \ln |\lambda_1|< -\frac{1}{2}\ln \left((1-\epsilon)\delta+\epsilon \right)     \label{eq:iffOracle1}\\
        \ln |\lambda_1|+ \ln |\lambda_2| < -\ln
        \left((1-\epsilon)\sqrt{\delta}+\epsilon
        \right)\label{eq:iffOracle2}
    \end{gather}
\end{theorem}

The following lemma is important in the proof of
Theorem~\ref{theorem:TwoDimensionalSystems}, which is stated first and
its proof can be found in the appendix.
\begin{lemma}~\label{lemma:OneTransmissionPeriod}
    If~\eqref{eq:iffOracle1} and~\eqref{eq:iffOracle2} are satisfied,
    with the scheduling Algorithm 2, we have that
    \begin{equation}
        \label{eq.desiredResult}
        \mathbb{E}\{\lambda_1^{2T_1^t}\delta^{n_1}\} <1, \quad  \E{ \lambda_2^{2T_1^t}\delta^{n_{2,1}}}<1
    \end{equation}
\end{lemma}
\begin{remark}
    Intuitively, Lemma~\ref{lemma:OneTransmissionPeriod} implies that
    with the designed scheduling Algorithm 2, the average expanding
    factor corresponding to the eigenvalues $\lambda_1$ and
    $ \lambda_2$ during one round transmission is smaller than one. In
    the proof of Theorem~\ref{theorem:TwoDimensionalSystems}, we will
    show that~\eqref{eq.desiredResult} is sufficient to ensure mean
    square stability.
\end{remark}
\textit{Proof of Theorem~\ref{theorem:TwoDimensionalSystems}:} Here
only the sufficiency is proved. The necessity follows directly
from~\eqref{eq.Necessity}. Define $T_0^t=0$, we have
\begin{align*}
  &  \E{ \sum_{t=1}^{\infty}(  \lambda_1^{2t}{\delta^{n_1^t}}+  \lambda_2^{2t}{\delta^{n_2^t}})}  \\
  & =\sum_{k=1}^{\infty} \E{ \sum_{j=1}^{T_{k+1}^t-1} (\lambda_1^{T_0^t+ \ldots + T_k^t+j}{\delta^{n_1^t}}+  \lambda_2^{T_0^t+ \ldots + T_k^t+j}{\delta^{n_2^t}})}  \\
  & \le \sum_{k=1}^{\infty} \E{ \sum_{j=1}^{T_{k+1}^t-1} (\lambda_1^{T_0^t+ \ldots + T_k^t+j}{\delta^{k n_1}}+  \lambda_2^{T_0^t+ \ldots + T_k^t+j}{\delta^{n_2^t}})}
\end{align*}

Since
\begin{align}
  & \sum_{k=1}^{\infty} \E{ \sum_{j=1}^{T_{k+1}^t-1} \lambda_1^{T_0^t+ \cdots + T_k^t+j}{\delta^{k n_1}  }} \nonumber \\
  & = \sum_{k=1}^{\infty}  \E{ \lambda_1^{T_0^t+\cdots+ T_k^t} \delta^{k n_1} \sum_{j=1}^{T_{k+1}^t-1}  } \nonumber\\
  &= \sum_{k=1}^{\infty} \E{ \frac{\lambda_1^{T_{k+1}^t}-\lambda_1^2}{\lambda_1^2-1} } \E{\lambda_1^{T_1^t} \delta^{n_1} }^k \label{eq:part1}
\end{align}
and
\begin{align}
  & \sum_{k=0}^{\infty} \E{ \sum_{j=1}^{T_{k+1}^t }  \lambda_2^{T_0^t+ \ldots + T_k^t+j}{\delta^{n_2^t}} } \nonumber \\
  & \le \sum_{k=0}^{\infty} \E{ \sum_{j=1}^{T_{k+1}^t }  \lambda_2^{T_0^t+ \ldots + T_k^t+j}{\delta^{n_{2,1}+\ldots + n_{2,k}}} } \nonumber \\
  &= \sum_{k=0}^{\infty} \E{ \lambda_2^{T_1^t } \delta^{n_{2, 1}} }^k   \E{ \frac{\lambda_2^{T_{k+1}^t}-\lambda_2^2}{\lambda_2^2-1} } \label{eq:part2}
\end{align}

In view of~\eqref{eq.desiredResult}, we know that~\eqref{eq:part1}
and~\eqref{eq:part2} are bounded. Thus
$\E{ \sum_{t=1}^{\infty}( \lambda_1^{2t}{\delta^{n_1^t}}+
  \lambda_2^{2t}{\delta^{n_2^t}})}$
is bounded, which further implies that
$ \lim_{t \rightarrow
  \infty}\E{\lambda_1^{2t}\delta^{n_1^t}+\lambda_2^{2t}\delta^{n_2^t}}
=0$. The proof of the sufficiency is complete. \hfill $\blacksquare$

\begin{remark}
    For $N$-dimensional systems, generally we want to minimize
    $\sum_{i=1}^N \lambda_i^{2t} \delta^{n_i^t}$ subject to the
    constraint that $\sum_{i=1}^{N}n_i^t=n$ with $n$ being the total
    number of successful transmissions by time $t$. The optimal choice
    of $n_i^t$ should be
    \begin{equation}
        \label{eq.goalForNStates}
        n_i^{t*}=\frac{1}{N} \left(  n+2t \frac{\sum_{i=1}^N \ln |\lambda_i|}{\ln \delta} \right) -2t \frac{\ln |\lambda_i|}{\ln \delta}
    \end{equation}
    However $n_i^{t*}$ is determined by $n$, which is not causally
    available when transmitting $x_{i,0}$ at any time $k<t$. When
    $N=2 $, we can achieve the desired optimal allocation by fixing
    $n_1^t=n_1$ and requiring $n_2^t$ to
    achieve~\eqref{eq:stoppingcondition}. However, this method is not
    applicable to the case of $N\ge 3$.
\end{remark}

\subsection{ An Example}
Suppose the parameters in the communication channel~\eqref{eq:channel}
are $P=1$, $\sigma_n^2=1$, $\epsilon=0.7$, the regions for
$(\ln{|\lambda_1|}, \ln{|\lambda_2|})$ indicated by the
necessity~\eqref{eq.Necessity}, the
sufficiency~\eqref{eq:TdmaSufficiency} with the TDMA scheduler, the
sufficiency~\eqref{eq:AdaptiveTDMASufficiency} with the adaptive TDMA
scheduler and the
sufficiency~\eqref{eq:iffOracle1}~\eqref{eq:iffOracle2} with the
optimal scheduler are plotted in Fig.~\ref{Fig.combine}.
\begin{figure}
    \centering
    \includegraphics[width=0.4\textwidth]{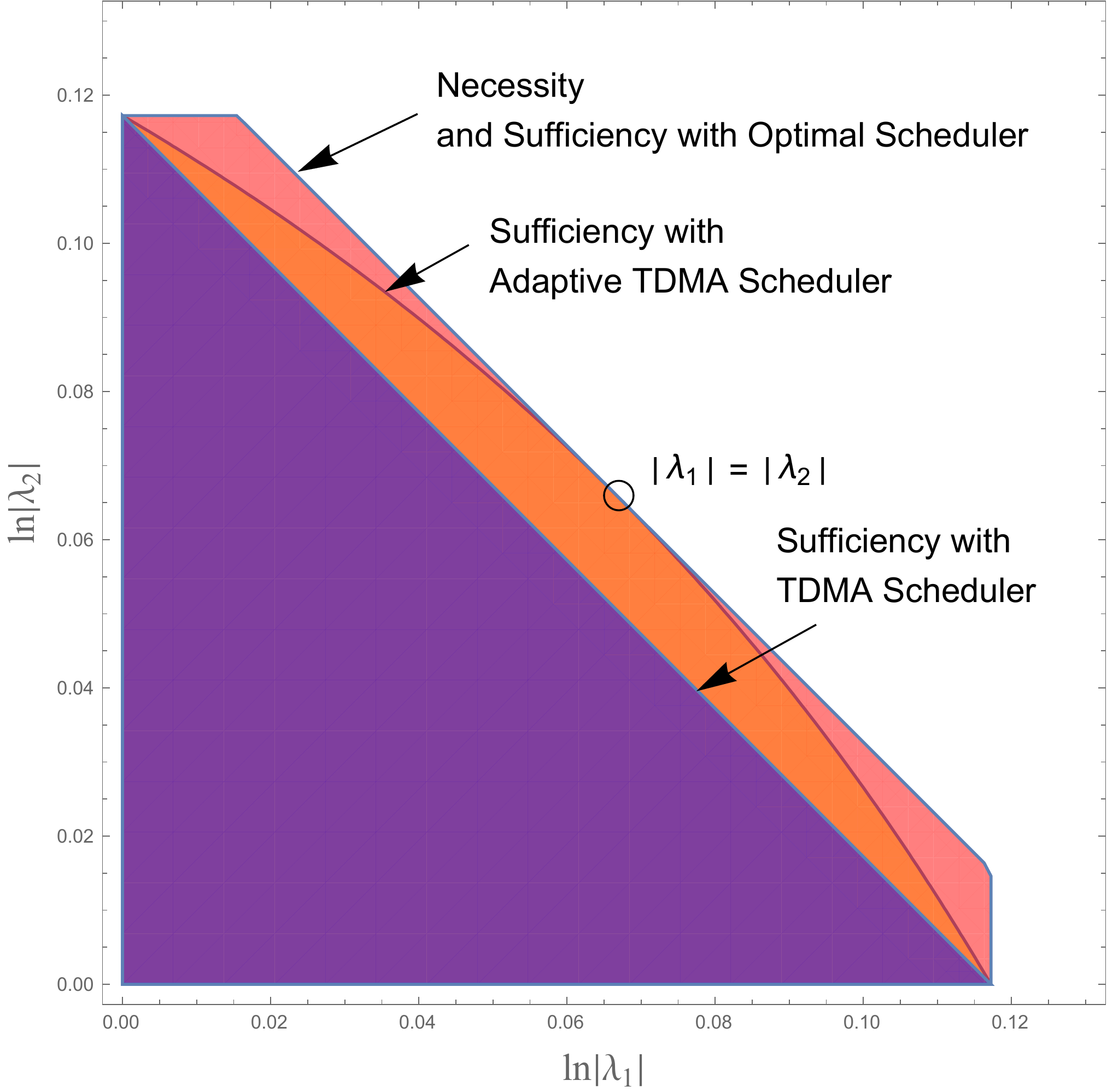}\\
    \caption{Comparisons of Stabilizability
      Conditions}\label{Fig.combine}
\end{figure}
It is clear from the figure that the optimal scheduler proposed in
Algorithm 2 covers the whole necessary stabilizability region, which
is larger than the regions that can be achieved by the adaptive and
conventional TDMA schedulers. Besides, as noted in
Remark~\ref{remark:adaptiveLargerThanConventional}, the adaptive TDMA
scheduler achieves a larger stabilizability region than that the
conventional TDMA scheduler. Moreover, we can observe that the
adaptive TDMA scheduler is optimal at three points, i.e.,
$|\lambda_1|=|\lambda_2|$, $|\lambda_1|=1$ and $|\lambda_2|=1$. This
is consistent with Corollary~\ref{corollary.equalMagnitude}.

\section{Conclusions}
This paper studies the mean square stabilizability problem of vector
LTI systems over power constrained lossy channels. Two transmission
schedulers are proposed and their stabilizability regions are
analyzed. It is shown that the proposed schedulers achieve larger
stabilizability regions than the one proposed in our previous work.
Further work will be devoted to the study of the optimal transmission
protocol for high-dimensional systems, and also for the case of
general power constrained fading channels.

\section*{Appendix}

Before stepping into the proof of
Lemma~\ref{lemma:OneTransmissionPeriod}, the following lemma is
needed.
\begin{lemma}
    \label{lemma:ThetaScope}
    If~\eqref{eq:iffOracle2} holds, the equation
    \begin{equation}
        \label{eq.equationOfTheta}
        \theta \phi-\ln [(1-\epsilon) \exp{\theta}+\epsilon]= 2\ln |\lambda_1|
    \end{equation}
    with
    $\phi = 2(\ln|\lambda_1| - \ln| \lambda_2| )/(\ln \delta) < 0$
    admits a unique solution $\theta$ with
    $ 0 > \theta > \frac12 \ln \delta$.
\end{lemma}

\begin{proof}
    Define the function
    $f(\theta)=\theta \phi-\ln [(1-\epsilon)
    \exp{\theta}+\epsilon]-2\ln |\lambda_1|$.
    Since $f$ is decreasing in $\theta$, and
    $f(0)=-2\ln |\lambda_1|<0$,
    $f(\frac{1}{2} \ln \delta) =-\ln
    |\lambda_1\lambda_2|[(1-\epsilon)\sqrt{\delta}+\epsilon]$.
    If~\eqref{eq:iffOracle2} holds, we have $f(\frac12 \ln \delta)>0$,
    which implies that~\eqref{eq.equationOfTheta} admits a unique
    solution and $ 0 > \theta > \frac12 \ln \delta $.
\end{proof}

\textit{Proof of Lemma \ref{lemma:OneTransmissionPeriod}:} In view of
the conditional expectation, at the time $t=T_1^1+T_1^2$, we have
\begin{align}
  &\mathbb{E}\{\lambda_1^{2(T_1^1+T_1^2)}\delta^{n_1^t}+\lambda_2^{2(T_1^1+T_1^2)}\delta^{n_2^t}\}\nonumber\\
  &  =\mathbb{E}\{ \mathbb{E}\{ \lambda_1^{2(T_1^1+T_1^2)}\delta^{n_1^t}+\lambda_2^{2(T_1^1+T_1^2)}\delta^{n_2^t}|T_1^1\le T^c\}\}\nonumber\\
  &  + \mathbb{E} \{ \mathbb{E}\{ \lambda_1^{2(T_1^1+T_1^2)}\delta^{n_1^t}+\lambda_2^{2(T_1^1+T_1^2)}\delta^{n_2^t}|T_1^1>
    T^c\} \} \nonumber\\
  &\overset{(a)}{\le } \mathbb{E} \{ \mathbb{E}\{ 2
    \lambda_1^{2(T_1^1+T_1^2)}\delta^{n_1}|T_1^1 \le T^c\}\}\nonumber \\
  &+\mathbb{E} \{ \mathbb{E}\{
    \lambda_1^{2T_1^1}\delta^{n_1}+\lambda_2^{2T_1^1}|T_1^1> T^c\}\} \label{eq.conditioning}
\end{align}
where $(a)$ follows from~\eqref{eq:stoppingcondition}.

Suppose $T_1^1$ is known and $T_1^1\le T^c$, with the definition of
$S_t = \sum_{i=T_1^1 + 1}^{T_1^1+t} \gamma_i$ and
$Y_t=\exp{\theta S_{t} + b t}$, we have
\begin{equation*}
    \E{Y_{t+1} |Y_t, Y_{t-1}, \ldots, Y_1} = Y_t \E{ \exp{\theta \gamma_{t+1} +b}}
\end{equation*}
Define $ b = - \ln \left[(1-\epsilon)\exp\theta + \epsilon\right]$, we
have
\begin{equation*}
    \E{ \exp{\theta \gamma_{t+1} +b}}=1
\end{equation*}
Thus the stochastic process $\{Y_t\}$ is a martingale. Since $T_1^2$
is a bounded stopping time, we can use the optional stopping
theorem~\cite{Ash2000Book} on $Y_t$, which yields
$\E{Y_{T_1^2} }= \E{Y_1}= 1$. However, by our stopping condition, we
know that
\begin{align*}
  S_{T_1^2} = n_2 = n_1 + 2(T_1^1+T_1^2)\times\frac{\ln |\lambda_1| - \ln |\lambda_2|}{\ln \delta}+c
\end{align*}
with $c \ge 0$. Therefore,
\begin{equation*}
    \E{ \exp{\theta n_1 + \theta \phi(T_1^1 + T_1^2)+\theta c + bT_1^2 }|T_1^1\le T^c }= 1
\end{equation*}
which implies that
\begin{align*}
  & \E{ \exp{(\theta\phi +b)T_1^2}|T_1^1\le T^c}  =\E{ \lambda_1^{2T_1^2}| T_1^1\le T^c} \\
  & = \exp{-\theta n_1 - \theta\phi T_1^1 -\theta c}
\end{align*}

In view of the above result and~\eqref{eq.conditioning}, we have
\begin{align}
  &\E{ \lambda_1^{2(T_1^1+T_1^2)}\delta^{n_1^t}+\lambda_2^{2(T_1^1+T_1^2)}\delta^{n_2^t}} \nonumber \\
  &\le \E{  \E{ \lambda_1^{2T_1^1}\delta^{n_1}+\lambda_2^{2T_1^1}|T_1^1> T^c}} \nonumber \\
  &+ \E{\E{ 2 \lambda_1^{2T_1^1} \exp{-\theta n_1-\theta \phi T_1^1-\theta c} \delta^{n_1}|T_1^1 \le T^c}} \nonumber \\
  &\le  \mathbb{E} \{  \E{ \lambda_1^{2T_1^1}\delta^{n_1}+\Omega|T_1^1> T^c}\}\nonumber \\
  &+ \E{ 2 \lambda_1^{2T_1^1} \exp{-\theta n_1-\theta \phi T_1^1-\theta c} \delta^{n_1}} \label{eq:eq26}
\end{align}
with
$\Omega := \lambda_2^{2T_1^1} -\delta^{n_1}2 \lambda_1^{2T_1^1}
\exp{-\theta n_1-\theta \phi T_1^1-\theta c }$.

In the following, we will show that when $T_1^1>T^c$, $\Omega <0$. We
only need to show that
$ \exp{2T_1^1\ln |\lambda_2|} <\mathrm{exp} (n_1 \ln \delta +2T_1^1
\ln |\lambda_1|+\ln 2-\theta n_1-\theta \phi T_1^1-\theta c) $
or equivalently
\begin{equation*}
    T_1^1(2\ln |\lambda_1| -\theta \phi -2\ln |\lambda_2|)>\theta
    n_1+\theta c -n_1 \ln \delta-\ln 2
\end{equation*}
If~\eqref{eq:iffOracle2} holds, in view of
Lemma~\ref{lemma:ThetaScope} we have $\theta> \ln \delta$, thus
$ 1-\frac{\theta}{\ln \delta}>0$, which means
$2(\ln |\lambda_1|-\ln |\lambda_2|) -\theta \phi >0$. Since
$T_1^1>T^c=-\frac{n_1}{\phi}$, we have
\begin{align*}
  T_1^1 (2\ln |\lambda_1| -\theta \phi -2\ln |\lambda_2|)& >  -\frac{n_1}{\phi}(2\ln |\lambda_1|-2 \ln |\lambda_2|) +\theta n_1\\
                                                         &\overset{(b)}{>} \theta n_1+ \theta c -n_1 \ln \delta-\ln 2
\end{align*}
where $(b)$ holds from the definition of $\phi$. Thus when
$T_1^1>T^c$, $\Omega <0$. From~\eqref{eq:eq26}, we have
\begin{align}
  &\E{ \lambda_1^{2(T_1^1+T_1^2)}\delta^{n_1^t}+\lambda_2^{2(T_1^1+T_1^2)}\delta^{n_2^t}} \nonumber\\
  % &\le \mathbb{E} \{ \E{
  % \lambda_1^{2T_1^1}\delta^{n_1}+\lambda_2^{2T_1^1}|T_1^1>
  % T^c}\}\nonumber\\
  % &+ \mathbb{E} \{\E{ 2 \lambda_1^{2T_1^1} \exp{-\theta n_1-\theta
  % \phi T_1^1-\theta c} \delta^{n_1}|T_1^1 \le T^c}\}\nonumber\\
  % &\le \mathbb{E} \{ \E{
  % \lambda_1^{2T_1^1}\delta^{n_1}+\Omega|T_1^1> T^c}\}\nonumber\\
  % &+ \E{ 2 \lambda_1^{2T_1^1} \exp{-\theta n_1-\theta \phi
  % T_1^1-\theta c} \delta^{n_1}}\nonumber\\
  &\le \E{ 2 \lambda_1^{2T_1^1} \exp{-\theta n_1-\theta \phi T_1^1-\theta c} \delta^{n_1}}+\E{ \lambda_1^{2T_1^1}\delta^{n_1}} \label{eq.upperBound}
\end{align}

For the first term in~\eqref{eq.upperBound}, we have
\begin{align*}
  & \E{ 2 \lambda_1^{2T_1^1} \exp{-\theta n_1-\theta \phi T_1^1-\theta c} \delta^{n_1}}\\
  &= 2\delta^{n_1}\exp{-\theta n_1 -\theta c}\times  \sum_{n_1}^{\infty} \lambda_1^{2T_1^1}\exp{- \theta\phi T_1^1} \text{Pr}(T_1^1)\\
  &= 2 \exp{-\theta c} \left(\delta \exp{-\theta}\times \frac{\lambda_1^2 \exp{-\theta\phi}(1-\epsilon)}{1-\lambda_1^2\exp{-\theta\phi}\epsilon}\right)^{n_1}
\end{align*}

In view of~\eqref{eq.equationOfTheta}, we have
\begin{equation*}
    \exp{-\theta\phi} =\frac{1}{\lambda_1^2\left[(1-\epsilon)\exp{\theta}+\epsilon\right]}
\end{equation*}
Therefore,
\begin{equation*}
    \delta \exp{-\theta}\times \frac{\lambda_1^2\exp{-\theta\phi}(1-\epsilon)}{1-\lambda_1^2\exp{-\theta\phi}\epsilon} = \delta \exp{-2\theta}
\end{equation*}

Besides for the second term in~\eqref{eq.upperBound}, we have
\begin{align*}
  \E{ \lambda_1^{2T_1^1}\delta^{n_1}} = \sum_{n_1}^{\infty}{   \lambda_1^{2T_1^1}\delta^{n_1} \text{Pr}(T_1^1)}  = \left( \frac{\lambda_1^2 \delta (1-\epsilon)}{1-\lambda_1^2\epsilon} \right)^{n_1}
\end{align*}

Thus
\begin{align*}
  &\E{ \lambda_1^{2(T_1^1+T_1^2)}\delta^{n_1^t}+\lambda_2^{2(T_1^1+T_1^2)}\delta^{n_2^t}}\\
  &\le 2 \exp{-\theta c}(\delta \exp{-2 \theta})^{n_1}+ \left(
    \frac{\lambda_1^2 \delta (1-\epsilon)}{1-\lambda_1^2\epsilon}
    \right)^{n_1}
\end{align*}
If~\eqref{eq:iffOracle1} holds, we have that
$ \frac{\lambda_1^2 \delta (1-\epsilon)}{1-\lambda_1^2\epsilon}<1$. If
\eqref{eq:iffOracle2} holds, in view of Lemma~\ref{lemma:ThetaScope},
we have that $\delta \mathrm{exp}(-2\theta)<1$. Thus by appropriately
selecting $n_1$, we can guarantee
$
\mathbb{E}\{\lambda_1^{2(T_1^1+T_1^2)}\delta^{n_1}+\lambda_2^{2(T_1^1+T_1^2)}\delta^{n_{2,1}}\}
<1 $,
which further ensures
$ \mathbb{E}\{\lambda_1^{2T_1^t}\delta^{n_1}\} <1 $ and
$\E{ \lambda_2^{2T_1^t}\delta^{n_{2,1}}}<1$. The proof is complete.
\hfill $\blacksquare$

\bibliographystyle{ieeetr}
\bibliography{references}

\end{document}